\documentclass[a4paper,11pt,twoside]{amsart}

\usepackage{amsfonts}
\usepackage{amsmath}
\usepackage{amsthm}
\usepackage{amssymb}
\usepackage[all]{xy}
\usepackage{hyperref}

\newtheorem{thm}{Theorem}[section]
\newtheorem{prop}[thm]{Proposition}
\newtheorem{lem}[thm]{Lemma}
\newtheorem{cor}[thm]{Corollary}

\numberwithin{equation}{section}

\theoremstyle{definition}
\newtheorem{definition}[thm]{Definition}
\newtheorem{remark}[thm]{Remark}

\renewenvironment{proof}{\par\vspace{-5pt}%
\par\noindent\begingroup%
\leftskip=0em\hspace{0em}{\bf Proof.}}%
{\endgroup\hfill$\Box$}

\DeclareMathOperator{\im}{im} % image of a morphism
\newcommand{\iso}{\cong}
\newcommand{\niso}{\ncong}

\newcommand{\farg}{-} % argument of a functor
\newcommand{\id}{\mathrm{id}}
\newcommand{\comp}{\circ} % composition
\newcommand{\mor}[1]{\xrightarrow{#1}}
\newcommand{\isomor}{\mor{\sim}} % isomorphism
\newcommand{\rest}[1]{|_{#1}} % restriction of function
\newcommand{\cp}[1]{#1^{\bullet}} % (cohomological) complex
\newcommand{\K}{\Bbbk} % field K
\newcommand{\cat}[1]{{\mathbf{#1}}} % category
\newcommand{\s}[1]{\mathcal{#1}} % sheaf (over a topological space)
\newcommand{\so}{\s{O}} % structure sheaf O
\newcommand{\Hom}{\mathrm{Hom}}

\newcommand{\End}{\mathrm{End}}

\newcommand{\Spec}{\mathrm{Spec}}
\newcommand{\cone}[1]{\mathrm{C}(#1)} % cone of a morphism
\newcommand{\rort}[1]{#1^{\perp}} % right orthogonal
\newcommand{\lort}[1]{{}^{\perp}#1} % left orthogonal
\newcommand{\sh}[2][1]{#2[#1]} % shift functor
\newcommand{\FM}[2][]{\Phi^{#1}_{#2}} % Fourier-Mukai functor

\newcommand{\D}[1][]{\mathbf{D}^{#1}} % derived category
\newcommand{\Db}{\mathrm{D}^b} % bounded derived category
\newcommand{\Dp}{\cat{Perf}} % category of perfect complexes
\newcommand{\rd}{\mathbf{R}} % right derived
\newcommand{\ld}{\mathbf{L}} % left derived
\newcommand{\lotimes}{\overset{\ld}{\otimes}} % der. fun. of tensor product
\newcommand{\fun}[1]{\mathsf{#1}} % functor
\newcommand{\Mod}[1]{\cat{Mod}\text{-}#1} % category of modules
\newcommand{\Coh}{\cat{coh}}
\newcommand{\Qcoh}{\cat{Qcoh}}
\newcommand{\ortdec}{\oplus} % orthogonal decomposition
\newcommand{\hocolim}{\underrightarrow{\operatorname{hocolim}}}
\newcommand{\inc}{\iota} % inclusion (in a direct sum)
\newcommand{\ch}{\mathrm{ch}}

\newcommand{\cal}{\mathcal}
\newcommand{\ka}{{\cal A}}

\newcommand{\ke}{{\cal E}}

\newcommand{\ko}{{\cal O}}
\newcommand{\kp}{{\cal P}}

\newcommand{\NN}{\mathbb{N}}
\newcommand{\ZZ}{\mathbb{Z}}
\newcommand{\QQ}{\mathbb{Q}}

\newcommand{\CC}{\mathbb{C}}

\begin{document}

	\title[Does full imply faithful?]{Does full imply faithful?}

    \author{Alberto Canonaco}
    \author{Dmitri Orlov}
    \author{Paolo Stellari}

	\address{A.C.: Dipartimento di Matematica ``F. Casorati'', Universit{\`a}
	degli Studi di Pavia, Via Ferrata 1, 27100 Pavia, Italy}
	\email{alberto.canonaco@unipv.it}

	\address{D.O.: Algebra Section, Steklov Mathematical Institute RAS, 8 Gubkin Str., Moscow 119991, Russia}
	\email{orlov@mi.ras.ru}

	\address{P.S.: Dipartimento di Matematica ``F.
	Enriques'', Universit{\`a} degli Studi di Milano, Via Cesare Saldini
	50, 20133 Milano, Italy}
	\email{paolo.stellari@unimi.it}


	\thanks{D.O.\ was partially supported by  RFFI grants
	08-01-00297, 10-01-93113,  NSh grant 4713.2010.1, and
	by AG Laboratory HSE, RF government
	grant, ag. 11.G34.31.0023. P.S.\ was partially supported by the
	MIUR of the Italian Government in the framework of the National Research Project ``Geometria algebrica e aritmetica, teorie coomologiche e teoria dei motivi'' (PRIN 2008).}

	
\begin{abstract}
We study full exact functors between triangulated categories. With some hypotheses on the source category we prove that it admits an orthogonal decomposition into two pieces such that the functor restricted to one of them is zero while the restriction to the other is faithful. In particular, if the source category is either the category of perfect complexes or the bounded derived category of coherent sheaves on a noetherian scheme supported on a closed connected subscheme, then any non-trivial exact full functor is faithful as well. Finally we show that removing the noetherian hypothesis this result is not true.
\end{abstract}

\maketitle

\section{Introduction}\label{Intro}

For an exact functor $\fun{F}\colon\cat{T}_1\to\cat{T}_2$ between triangulated categories there is a list of properties that, from a purely categorical point of view, are completely unrelated or not automatically satisfied. Among them we can mention: the existence of adjoints, fullness, faithfulness and essential surjectivity. Nevertheless, as soon as $\cat{T}_i$ has a geometric nature, these properties and their relations can be studied in a more efficient and complete way.

For example, if $\cat{T}_i$ is the bounded derived category $\Db(X_i)$
of coherent sheaves on a complex smooth projective variety $X_i,$ then
any exact functor $\fun{F}\colon\Db(X_1)\to\Db(X_2)$ has always a left
and a right adjoint, by a result of Bondal and Van den Bergh \cite{BB}. This, combined with  \cite{Or1}, says that if $\fun{F}$ is fully faithful, then it is of \emph{Fourier--Mukai type}, i.e.\ there is $\ke\in\Db(X_1\times X_2)$ and an isomorphism of functors $\fun{F}\iso\FM{\ke},$ where
	$\FM{\ke}:\Db(X_1)\to\Db(X_2)$ is the exact functor defined by
	\[
	\FM{\ke}:=\rd(p_2)_*(\ke\lotimes p_1^*(-)),
	\]
and $p_i:X_1\times X_2\to X_i$ is the natural projection.

Now \cite{BO} and \cite{Br} provide a very useful criterion to establish when a Fourier--Mukai functor $\FM{\ke}\colon\Db(X_1)\to\Db(X_2)$ is fully faithful. Namely $\FM{\ke}$ is such if and only if
\[
\Hom_{\Db(X_2)}(\FM{\ke}(\ko_{x_1}),\sh[i]{\FM{\ke}(\ko_{x_2})})\iso\begin{cases}\CC & \text{if $x_1=x_2$ and $i=0$}\\0 & \text{if $x_1\neq x_2$ or $i\not\in[0,\dim X_1]$}\end{cases}
\]
for all closed points $x_1,x_2\in X_1.$

Of course, it is quite easy to construct examples of faithful functors
which are not full (e.g.\ the tensorization by a vector bundle of rank
greater than 1). On the other hand, using all the remarks above and
a collection of standard results, it is not difficult to see that a
non-trivial full exact functor $\fun{F}\colon\Db(X_1)\to\Db(X_2)$ is
also faithful. Here we give a sketch of the proof, since a more
general statement will be proved in the paper. Firstly, by the main
result of \cite{CS} (which improves \cite{Or1}), $\fun{F}$ is a
Fourier--Mukai functor. Thus, because of the above criterion and the
fact that $\fun{F}$ is full, to show that the functor is also faithful
it is enough to prove that there are no closed points $x\in X_1$ such
that $\Hom(\fun{F}(\ko_x),\fun{F}(\ko_x))=0$ or, in other words, such
that $\fun{F}(\ko_x)\iso 0.$ To see this, take the left adjoint
$\fun{G}\colon\Db(X_2)\to\Db(X_1)$ of $\fun{F}$ and consider the
composition $\fun{G}\comp\fun{F}$ which is again a Fourier--Mukai
functor, hence isomorphic to $\FM{\ke}$ for some $\ke\in\Db(X_1\times
X_1).$ Assume that there are $x_1,x_2\in X_1$ such that
$\fun{F}(\ko_{x_1})\not\iso 0$ while $\fun{F}(\ko_{x_2})\iso 0.$ By
\cite{BO} (see, in particular, Proposition 1.5 there) the Chern
character $\ch(\FM{\ke}(\ko_{x_1}))$ is not zero. On the other hand,
it is proved in \cite{Or1} that the functor $\FM{\ke}$
induces a morphism $\Phi_\ke^H\colon H^*(X_1,\QQ)\to H^*(X_1,\QQ)$
such that
\[
0\neq\ch(\FM{\ke}(\ko_{x_1}))\cdot\sqrt{\mathrm{td}(X_2)}=\Phi_\ke^H(\ch(\ko_{x_1})\cdot\sqrt{\mathrm{td}(X_1)})=\Phi_\ke^H(\ch(\ko_{x_2})\cdot\sqrt{\mathrm{td}(X_1)})=0.
\]
This contradiction proves that, if $\fun{F}$ were not faithful, then
$\fun{F}(\ko_x)\iso0$ for every closed point $x\in X.$ But this would
easily imply that $\fun{F}\iso0,$ against the assumption.

This paper is an attempt to understand to which extent the previous easy example can be pushed. In particular, we want to study when the following question may have a positive answer:
\medskip

\centerline{\sf When is a full exact functor between `geometric triangulated categories' faithful?}
\medskip

\noindent It is rather obvious that one can produce examples of full
non-trivial exact functors which are not faithful if one does not
require the source triangulated category to be
indecomposable. However, something interesting can be said even without
this hypothesis. In fact, after proving a very general statement in
Section \ref{subsec:general}, our first important (and still rather
general) result, whose proof is in Section \ref{sec:proof}, is the
following.

\begin{thm}\label{thm:gen2}
Let $\cat{T}_1$ be a triangulated category with arbitrary direct sums
that is compactly generated and let $\cat{T}_1^c$ be the subcategory of compact objects.
 Let $S\subset\cat{T}_1^c$ be a subset of compact objects and let $\cat{S}\subseteq\cat{T}_1^c$
 be the thick subcategory generated by $S.$ Let
\[
\fun{F}\colon\cat{S}\longrightarrow\cat{T}_2
\]
be a full exact functor to a triangulated category $\cat{T}_2.$ Assume
that for any object $A\in S$ the ring of endomorphisms
$\End_{\cat{T}_1}(A)$ is idempotent noetherian. Then there is an
orthogonal decomposition
\[
\cat{S}=\rort{(\ker\fun{F})}\ortdec\ker\fun{F}
\]
and $\fun{F}\rest{\rort{(\ker\fun{F})}}$ is faithful.
\end{thm}

See Definition \ref{def:idempnoeth} for the notion of idempotent
noetherian ring. As it will turn out, the ring of endomorphisms of an
object in the bounded derived category of coherent sheaves on a
noetherian scheme has this property (see Proposition
\ref{prop:noeth}).

Notice that if in Theorem \ref{thm:gen2} we assume $\cat{S}$ to be
indecomposable and $\fun{F}$ to be non-trivial, then we can conclude
that $\fun{F}$ is actually faithful. So in the geometric case we
consider a noetherian scheme $X$ containing a closed connected
subscheme $Z$ and we assume that $\cat{S}$ is either the bounded
derived category $\Db_Z(X)$ of coherent sheaves on $X$ supported on
$Z$ or the subcategory $\Dp_Z(X)\subseteq\Db_Z(X)$ consisting of
perfect complexes. Recall that a complex in $\Db_Z(X)$ is
\emph{perfect} if it is locally quasi-isomorphic to a complex of
locally free sheaves of finite type on $X.$ Due to Corollary
\ref{cor:orth}, these categories are indecomposable, and we get the
following result which we prove in Section \ref{subsec:geometry}.

\begin{thm}\label{thm:main}
Let $X$ be a noetherian scheme containing a closed subscheme $Z$ and
let $\cat{S}$ be either $\Dp_Z(X)$ or $\Db_Z(X).$ Let $\cat{T}$ be a
triangulated category and let
\[
\fun{F}\colon\cat{S}\longrightarrow\cat{T}
\]
be a full exact functor which is not isomorphic to the zero functor. If $Z$ is connected, then $\fun{F}$ is also faithful.	
\end{thm}

In Section \ref{sec:counter} we show that if we do not assume $X$ to
be noetherian, then the above result does not necessarily hold
true. Indeed, we give an example of a non-noetherian (affine) scheme
$X$ over a field $\K$ such that $\Dp(X)$ is indecomposable and of a
full non-trivial exact functor $\fun{F}\colon\Dp(X)\to\D(\K)$ to the
(unbounded) derived category of $\K$-vector spaces which is not
faithful.

\section{A general result}\label{subsec:general}

If $\fun{F}\colon\cat{A}\to\cat{B}$ is an
additive functor between additive categories, we will denote by
$\ker\fun{F}$ the (strictly) full subcategory of $\cat{A}$ having as
objects the $A$ such that $\fun{F}(A)\iso0,$ and by $\im\fun{F}$ the
(strictly) full subcategory of $\cat{B}$ having as objects the $B$
such that $B\iso\fun{F}(A)$ for some $A\in\cat{A}.$ Notice that
$\ker\fun{F}$ is a (thick) triangulated subcategory of $\cat{A}$ if
$\cat{A}$ and $\cat{B}$ are triangulated and $\fun{F}$ is exact.

For the convenience of the reader we recall the proof of the following lemma which is known to experts and, for example, is contained
in the proof of \cite[Thm.\ 3.9]{Or}.

\begin{lem}\label{lem:faithful}
Let $\cat{T}_1$ and $\cat{T}_2$ be triangulated categories and let $\fun{F}:\cat{T}_1\to\cat{T}_2$ be a full exact functor such that $\ker\fun{F}\iso 0.$ Then $\fun{F}$ is faithful.
\end{lem}

\begin{proof}
Assume that there are $A,B\in\cat{T}_1$ and $f:A\to B$ a morphism such that $\fun{F}(f)=0.$ Complete the morphism to a distinguished triangle
\[
\xymatrix{
A\ar[r]^-{f}& B\ar[r]^-{g}&\cone{f}
}
\]
so that, applying the functor $\fun{F},$ we get the distinguished triangle
\[
\xymatrix{
\fun{F}(A)\ar[rr]^-{\fun{F}(f)=0}&&\fun{F}(B)\ar[rr]^-{\fun{F}(g)}&&\fun{F}(\cone{f}).
}
\]
Then $\id:\fun{F}(B)\to\fun{F}(B)$ factors through $\fun{F}(g)\colon\fun{F}(B)\to\fun{F}(\cone{f}).$

As $\fun{F}$ is full, there exists a morphism $h\colon B\to B$ factoring through $g$ and such that $\fun{F}(h)=\id.$ Then $\fun{F}(\cone{h})\iso\cone{\fun{F}(h)}\iso 0.$ Since $\ker\fun{F}\iso 0,$ we get $\cone{h}\iso 0$ and $h$ is an isomorphism. This implies that $g$ is a (split) monomorphism. In particular $f=0,$ and so $\fun{F}$ is faithful.
\end{proof}

\begin{definition}
An {\em orthogonal decomposition} $\cat{T}=\cat{T}_1\ortdec\cat{T}_2$
of a triangulated category $\cat{T}$ is given by two full triangulated
subcategories $\cat{T}_1$ and $\cat{T}_2$ satisfying the following
conditions:
\begin{itemize}
\item[(1)] $\cat{T}_1$ and $\cat{T}_2$ are completely orthogonal, meaning that
$\Hom(A_1,A_2)=\Hom(A_2,A_1)=0$ for every objects $A_i$ of
$\cat{T}_i;$
\item[(2)] For every object $A$ of $\cat{T}$ there exist objects $A_i$ of
$\cat{T}_i$ such that $A\iso A_1\oplus A_2.$
\end{itemize}
A triangulated category is {\em indecomposable} if it admits only
trivial orthogonal decompositions.
\end{definition}

We begin with the following general result.

\begin{prop}\label{prop:gen1}
Let $\cat{T}_1$ and $\cat{T}_2$ be triangulated categories and let $\fun{F}:\cat{T}_1\to\cat{T}_2$ be a full exact functor. Assume moreover that the projection functor $\pi\colon\cat{T}_1\to\cat{T}_1/\ker\fun{F}$ has an adjoint $\mu\colon\cat{T}_1/\ker\fun{F}\to\cat{T}_1.$ Then the category $\cat{T}_1$ has an orthogonal decomposition of the form
\[
\cat{T}_1=\im\mu\ortdec\ker\fun{F}
\]
and $\fun{F}\rest{\im\mu}$ is faithful. In particular, if $\cat{T}_1$
is indecomposable and $\fun{F}$ is not isomorphic to $0,$ then $\fun{F}$
is faithful.
\end{prop}

\begin{proof}
Passing, if necessary, to the opposed functor of $\fun{F}$ (defined as $\fun{F},$ but between the opposed categories), we can assume that $\mu$ is a right adjoint of $\pi.$

Now, given $A\in\cat{T}_1$ and using the adjunction between $\mu$ and $\pi,$ we get the distinguished triangle
\begin{equation*}\label{eqn:1}
\xymatrix{
A\ar[r]^-{m_A}& \mu\comp\pi(A)\ar[r]^-{n_A}& N_A.
}
\end{equation*}
The functor $\fun{F}$ induces in a natural way a functor $\fun{F}'\colon\cat{T}_1/\ker\fun{F}\to\cat{T}_2$ which is fully faithful due to Lemma \ref{lem:faithful}. Hence, for all $A,B\in\cat{T}_1,$
\[
\Hom(B,\mu\comp\pi(A))\iso\Hom(\pi(B),\pi(A))\iso\Hom(\fun{F}'\comp\pi(B),\fun{F}'\comp\pi(A))=\Hom(\fun{F}(B),\fun{F}(A)).
\]

As $\fun{F}=\fun{F}'\comp\pi$ is full, this implies that the morphism
\[
\Hom(B,A)\longrightarrow\Hom(B,\mu\comp\pi(A))
\]
given by the composition with $m_A$ is surjective for all $A,B\in\cat{T}_1.$ In particular, the map
\[
\varphi_{A,B}\colon\Hom(B,\mu\comp\pi(A))\longrightarrow\Hom(B,N_A),
\]
obtained composing with $n_A,$ is zero. Taking $B=\mu\comp\pi(A)$ in the above argument, we get $\varphi_{A,B}(\id)=n_A=0.$ This means that, for any $A\in\cat{T}_1,$ there is a decomposition
\[
A\iso\mu\comp\pi(A)\oplus\sh[-1]{N_A}.
\]

By \cite[Lemma 9.1.7]{N2} the functor $\mu$ as adjoint to a projection functor is fully faithful, i.e. $\pi\circ\mu\cong \id.$
Therefore, the functor $\pi$ induces an equivalence between $\im\mu$ and  the quotient $\cat{T}_1/\ker\fun{F}.$
Since $\fun{F}'$ is faithful, the functor $\fun{F}\rest{\im\mu}$ is faithful too.

Moreover, since $\mu$ is fully faithful the map
$\pi(m_A)$ is an isomorphism.
This implies that $\pi(N_A)\iso\cone{\pi(m_A)}\iso 0.$
In order to get the orthogonal decomposition, it remains to show that $\ker\fun{F}$ and $\im\mu=\im(\mu\comp\pi)$ are orthogonal. By adjunction, it is obvious that $\Hom(A,B)=0$ if $A\in\ker\fun{F}$ and $B\in\im\mu.$ For the other direction, assume that there is a morphism $f\colon\mu\comp\pi(A)\to B,$ for some $A\in\cat{T}_1$ and $B\in\ker\fun{F}.$ Consider the distinguished triangle
\[
\xymatrix{
\mu\comp\pi(A)\ar[r]^-{f}& B\ar[r]&\cone{f}
}
\]
and apply the functor $\pi$ getting
\[
\xymatrix{
\pi\comp\mu\comp\pi(A)\ar[r]^-{\pi(f)}& \pi(B)\ar[r]&\pi(\cone{f}).
}
\]
Thus $\pi(\sh[-1]{\cone{f}})\iso\pi\comp\mu\comp\pi(A)\iso\pi(A)$ and
$\mu\comp\pi(\sh[-1]{\cone{f}})\iso\mu\comp\pi(A).$ Moreover the map
$\sh[-1]{\cone{f}}\to\mu\comp\pi(A)$ can be identified with the canonical map $\sh[-1]{\cone{f}}\to\mu\comp\pi(\sh[-1]{\cone{f}}).$

As $\cone{f}\in\cat{T}_1,$ the calculations above yield that the map
$\sh[-1]{\cone{f}}\to\mu\comp\pi(\sh[-1]{\cone{f}})\iso\mu\comp\pi(A)$
is an epimorphism, and so $f=0.$ This is what we need to prove.
%Since $\fun{F}=\fun{F}'\comp\pi$ and $\fun{F}'$ is faithful, to show that $\fun{F}\rest{\im\mu}$ is faithful it is enough to prove that
%$\pi\rest{\im\mu}$ is faithful. But, due to the above calculation, for any $A\in\cat{T}_1,$
%$\mu\comp\pi(m_A)\colon\mu\comp\pi(A)\to\mu\comp\pi\comp\mu\comp\pi(A)$ is an isomorphism and
%so $\mu\comp\pi\rest{\im\mu}\iso\id.$ In particular, $\mu\comp\pi\rest{\im\mu}$ is faithful, whence the same is true for $\pi\rest{\im\mu}.$
\end{proof}

\begin{remark}\label{rmk:sat}
It is well-known that every exact functor from $\cat{T}_1$ has a right (respectively left) adjoint if $\cat{T}_1$ is right (respectively left) saturated (see \cite{BB}).
\end{remark}

\begin{remark}\label{rmk:pseudoadj}
Assume that $\cat{T}_1$ and $\cat{T}_2$ are triangulated categories and let $\fun{F}:\cat{T}_1\to\cat{T}_2$ be a full exact functor admitting a pseudo-adjoint $\fun{G}\colon\cat{T}_2\to\tilde{\cat{T}}_1$ such that $\im(\fun{G}\comp\fun{F})\subseteq\cat{T}_1.$ Then $\pi$ has an adjoint which is simply $\fun{G}\comp\fun{F}'$ (where $\fun{F}'\colon\cat{T}_1/\ker\fun{F}\to\cat{T}_2$ is defined in the above proof). Hence Proposition \ref{prop:gen1} applies.

With a  \emph{left} (respectively \emph{right}) {\em pseudo-adjoint} of a functor
$\fun{F}\colon\cat{C}\to\cat{C}'$ we mean a functor
$\fun{G}\colon\cat{C}'\to\tilde{\cat{C}},$ where $\tilde{\cat{C}}$ is
some category containing $\cat{C}$ as a full subcategory, together
with a natural isomorphism
$\Hom_{\cat{C}'}(A',\fun{F}(A))\iso\Hom_{\tilde{\cat{C}}}(\fun{G}(A'),A)$
(respectively
$\Hom_{\cat{C}'}(\fun{F}(A),A')\iso\Hom_{\tilde{\cat{C}}}(A,\fun{G}(A'))$)
for all objects $A$ of $\cat{C}$ and $A'$ of $\cat{C}'.$
\end{remark}

\section{The categorical case}\label{sec:proof}

In this section we prove Theorem \ref{thm:gen2} and show how to apply it to subcategories of  noetherian objects.
For this purpose we introduce the notion of an idempotent noetherian ring.

\subsection{General setting}\label{subsec:categprelim}

We will be interested in the following special class of rings
appearing naturally in geometric situations.

\begin{definition}\label{def:idempnoeth}
A ring $R$ is \emph{(right) idempotent noetherian} if for every sequence
$\{a_i\}_{i\in\NN}$ of elements in $R$ satisfying
\begin{equation}\label{eqn:strangeseq}
a_ja_i=a_i\qquad\text{for all $i<j$}
\end{equation}
there exists a positive integer $n$ such that $a_iR=a_nR$ for all $i\geq n.$
\end{definition}

Analogously, one can define left idempotent noetherian rings. As this notion will not be needed in the rest of the paper, right idempotent noetherian rings will simply be called idempotent noetherian.

\begin{remark}\label{rmk:idem}
If $\{a_i\}_{i\in\NN}$ is a sequence in a ring $R$ satisfying
\eqref{eqn:strangeseq} and such that $a_iR=a_nR$ for $i\geq n,$ then
$a_i$ is idempotent for $i>n.$ Indeed, there exists $r\in R$ such that
$a_i=a_{i-1}r,$ hence
\[a_ia_i=a_ia_{i-1}r=a_{i-1}r=a_i.\]
\end{remark}

We begin with the following easy result.

\begin{lem}\label{lem:useful}
If $\cat{A}$ is an additive category and $A\in\cat{A}$ is such that $\End_{\cat{A}}(A)$ is idempotent noetherian, then $A$ is isomorphic to a finite direct sum of indecomposable objects.
\end{lem}

\begin{proof}
Assume on the contrary that $A$ is not isomorphic to a finite direct sum of indecomposable objects. Then there exists a sequence $\{A_i\}_{i\in\NN}$ of non-trivial objects of $\cat{A}$ such that, for all $n\in\NN,$ the object $B_n:=\bigoplus_{j=0}^n A_j$ is a direct summand of $A.$ Thus, for all $j\in\NN$ set $a_j\in\End(A)$ to be the projection onto $B_j.$ Clearly the sequence $\{a_i\}_{i\in\NN}$ satisfies \eqref{eqn:strangeseq} but the ascending chain of right ideals generated by the $a_i$'s does not stabilize.
\end{proof}

As a matter of notation, recall that if $\cat{T}$ is a
triangulated category with arbitrary direct sums and $S$ is a
set of objects of $\cat{T},$ the {\em localizing subcategory
  generated by $S$} is the smallest strictly full triangulated
subcategory of $\cat{T}$ containing $S$ and closed under
arbitrary direct sums.

An object $A$ in a triangulated category $\cat{T}$ admitting arbitrary direct sums is called \emph{compact} if for each family of objects $\{X_i\}_{i\in I}\subset\cat{T}$ the canonical map
\[
\bigoplus_i\Hom(A,X_i)\longrightarrow\Hom\left(A,\oplus_i X_i\right)
\]
is an isomorphism. The triangulated category $\cat{T}$ is \emph{compactly generated} if there is a a set $S$ of compact objects such that $E\in\cat{T}$ vanishes if $\Hom(A,\sh[i]{E})=0$ for all $A\in S$ and all $i\in\ZZ.$ For more details, the reader can consult \cite[Sect.\ 3.1]{R}.

\medskip

\noindent
{\bf Proof of Theorem \ref{thm:gen2}.}
Denote by $\langle S\rangle\subseteq\cat{T}_1$ the localizing subcategory
generated by the set $S.$ This category admits arbitrary direct sums
and is compactly generated too.
Moreover, it is known that the subcategory of its compact objects $\langle S\rangle^c$ coincides with $\cat{S}$ (see \cite[Lemma 2.2]{N1}).
Hence, replacing $\cat{T}_1$ with $\langle S\rangle$ we can assume that $\cat{T}_1$ is compactly generated by the set $S$ and
 $\cat{S}=\cat{T}_1^c.$

Denote by $\langle \ker\fun{F}\rangle \subseteq\cat{T}_1$ the localizing subcategory that
is generated by the set of compact objects from $\ker\fun{F}.$ By
\cite[Thm.\ 2.1]{N1}, the canonical functor
$\cat{T}_1^c/\ker\fun{F}\to\cat{T}_1/\langle \ker\fun{F}\rangle$ is fully faithful and
its essential image is the subcategory $(\cat{T}_1/\langle \ker\fun{F}\rangle)^c.$ As
$\cat{T}_1$ is compactly generated the projection
$\pi\colon\cat{T}_1\to\cat{T}_1/\langle \ker\fun{F}\rangle$ has a fully faithful right
adjoint $\mu\colon\cat{T}_1/\langle \ker\fun{F}\rangle\to\cat{T}_1$ (see Theorem 8.4.4
and Lemma 9.1.7 in \cite{N2}).

In view of Proposition \ref{prop:gen1}, the result is proved if
$\mu\comp\pi(A)$ is compact, for any compact $A\in\cat{T}_1^c.$ Since
$\cat{T}_1^c$ is the smallest thick subcategory containing $S$ it is enough
to prove that $\mu\comp\pi(A)\in\cat{T}_1^c,$ for any
$A\in S.$ By Lemma \ref{lem:useful}, we can assume that $A$ is indecomposable.

Consider the adjunction morphism $m_A\colon A\to\mu\comp\pi(A)$ and
complete it to a distinguished triangle
\begin{equation}\label{trian}
\xymatrix{
\sh[-1]{N_A}\ar[r]^-{l_A}&A\ar[r]^-{m_A}&\mu\comp\pi(A)\ar[r]^-{n_A}&N_A.
}
\end{equation}
Of course, the result is proved if we show that $n_A$ is the zero
map, whence we can assume that $N_A\niso0.$

The functor $\fun{F}$ is full and so, by the same argument as in the proof of Proposition \ref{prop:gen1}, the map
$\Hom_{\cat{T}_1}(B,A)\mor{m_A\comp(-)}\Hom_{\cat{T}_1}(B,\mu\comp\pi(A))$
is surjective for any compact object $B\in\cat{T}_1^c.$ This implies that
the map
\begin{equation}\label{eqn:us1}
\Hom_{\cat{T}_1}(B,\mu\comp\pi(A))\mor{n_A\comp(-)}\Hom_{\cat{T}_1}(B,N_A)
\end{equation}
is zero.

Since $\cat{T}_1$ is compactly generated, there exists
$Z\in\cat{T}_1^c$ and a non-trivial morphism $\phi_0\colon Z\to\sh[-1]{N_A}.$
Denote by $C_Z$ the cone in $\cat{T}_1^c$ of the morphism
$l_A\comp\phi_0\colon Z\to A$ and consider the following commutative
diagram whose rows are distinguished triangles
\[
\xymatrix{
\sh[-1]{C_Z} \ar[d]\ar[rr] && Z \ar[d]^{\phi_0}\ar[rr]^-{l_A\comp\phi_0}  && A\ar[d]^{\id}\ar[rr]&& C_Z \ar[d]\\
\sh[-1]{\mu\comp\pi(A)} \ar[rr]^-{\sh[-1]{-n_A}} && \sh[-1]{N_A}\ar[rr]^-{l_A} && A\ar[rr]^-{m_A} && \mu\comp\pi(A).
}
\]

As $C_Z$ is a compact object, the composition map
$\sh[-1]{C_Z}\to\sh[-1]{\mu\comp\pi(A)}\to\sh[-1]{N_A}$ is the zero morphism
(use that the morphism in (\ref{eqn:us1}) is trivial).  Hence there is a
non-trivial map $\phi_1\colon A\to\sh[-1]{N_A}$ such that $\phi_1\comp
l_A\comp\phi_0=\phi_0.$ Now consider $A$ and $\phi_1$ instead of the
pair $Z$ and $\phi_0.$ Repeating the same argument as above we obtain
another map $\phi_2\colon A\to\sh[-1]{N_A}$ such that $\phi_2\comp
l_A\comp\phi_1=\phi_1.$ In conclusion, this procedure yields a
sequence of morphisms $\phi_i\colon A\to\sh[-1]{N_A}$ such that
$\phi_{i+1}\comp l_A\comp\phi_i=\phi_i,$ for $i>0.$

Set $a_i:=l_A\comp\phi_i,$ for any $i>0.$ This defines a sequence satisfying \eqref{eqn:strangeseq} in
$\End(A).$ But by
assumption this ring is idempotent noetherian. Hence there exists
$n\in\NN$ such that $a_i\comp\End(A)=a_n\comp\End(A),$ for all $i\geq n.$ Given $N>n,$ by
Remark \ref{rmk:idem} $a_N$ is idempotent. Since
$a_N=l_A\comp\phi_N$ is not zero and $A$ is indecomposable, $a_N$ must be the identity and so $A$ is a direct summand of $\sh[-1]{N_A}.$ This implies $m_A=0.$ Since $m_A$ corresponds to $\id_{\pi(A)}$ by adjunction, this means $\pi(A)\iso 0$ and so $\mu\comp\pi(A)\iso 0$ as well. This concludes the proof of Theorem \ref{thm:gen2}. \hfill $\Box$

\begin{remark}\label{rmk:keller}
It is important to note that the theorem above can be applied to a large class of triangulated categories.
Assume that our triangulated category $\cat{S}$ is algebraic, i.e.\ it can be realized as a homotopy category of some differential graded category.
If $\cat{S}$ is idempotent complete and equals to the closure of  a set of objects $S$ under shifts, extensions
and passage to direct factors (i.e.\ classically generated by this set),
then by part b) of \cite[Thm.\ 3.8]{K} the category $\cat{S}$ is equivalent to a category of compact objects in the derived category of a dg-category, which is compactly generated and admits arbitrary direct sums.
Thus it follows that if the rings of endomorphisms of all objects from $S$ are idempotent noetherian,
then the statement of Theorem \ref{thm:gen2} holds for such $\cat{S}.$
\end{remark}

\subsection{Derived categories of abelian categories}\label{subsec:intermed}
Recall that an object $E$ in an abelian category is called
\emph{noetherian} if any ascending chain $G_1\subseteq
G_2\subseteq\ldots\subseteq G_n\subseteq\ldots\subseteq E$ of
subobjects of $E$ stabilizes, i.e.\ there is $n\in\NN$ such that
$G_n=G_i$ for all $i\ge n.$ An abelian category is called \emph{noetherian} if it is equivalent to a small category and every object is noetherian. An abelian category is called \emph{locally noetherian} if it satisfies axiom (AB5) and has a set of noetherian generators
(see, for example, \cite{Po}).

\begin{remark}\label{rmk:locallynoeth}
It can be proved that the full subcategory of noetherian objects in any locally noetherian
abelian category is itself a noetherian abelian category.
\end{remark}

The following statement says that the endomorphism algebra of  a `noetherian' object is idempotent noetherian.

\begin{prop}\label{prop:idnoeth}
Let $\cat{A}$ be an abelian category with countable direct sums. Let
$C\in\D[b](\cat{A})$ be an object such that the cohomology
$H^k(C)\in\cat{A}$ is noetherian
for every $k\in\ZZ.$ Then the algebra $\End_{\D[b](\cat{A})}(C)$ is idempotent noetherian.
\end{prop}

\begin{proof}
Let $\{a_i\}_{i\in\NN}$ be a sequence in $\End(C)$ satisfying
\eqref{eqn:strangeseq}. We set $M:=\oplus_{i\in\NN}C$ and
$N:=\hocolim\{a_i\},$ so that there is a distinguished triangle in
$\D[b](\cat{A})$
\begin{equation}\label{eqn:distr}
\xymatrix{
M \ar[r]^-{f} & M \ar[r]^-{a'} & N}
\end{equation}
where, denoting by $\inc_i\colon
C\to M$ (for $i\in\NN$) the inclusion of the $i^{\text{th}}$
component, the morphism $f$ is defined by
$f\comp\inc_i:=\inc_i-\inc_{i+1}\comp a_i.$ By \eqref{eqn:strangeseq} the
morphism $a\colon M\to C$ defined by $a\comp\inc_i:=a_i$ clearly satisfies
$a\comp f=0,$ hence there exists a morphism $b\colon N\to C$ such that $b\comp a'=a.$ Then, setting also $a'_i:=a'\comp\inc_i\colon C\to N,$ we have
\begin{equation}\label{eqn:comp}
b\comp a'_i=a_i\qquad\text{for all $i\in\NN$}.
\end{equation}
Observe that, if $i\in\NN$ is such that $a'_i\comp b\colon N\to N$ is an
isomorphism, then $a_i\comp \End(C)=b\comp \Hom(C,N).$ Indeed, by \eqref{eqn:comp}
we have $a_i\comp c=b\comp a'_i\comp c$ for every $c\in\End(C).$ Conversely, \eqref{eqn:comp} implies that
\[
b\comp d=b\comp (a'_i\comp b)\comp (a'_i\comp b)^{-1}\comp d=a_i\comp b\comp (a'_i\comp b)^{-1}\comp d
\]
for every $d\in\Hom(C,N).$

Thus, in order to conclude that $\End(C)$ is idempotent noetherian,
it is enough to prove that for $i\gg0$ the morphism $a'_i\comp b$ is an
isomorphism in $\D[b](\cat{A}),$ which is the case if and only if
$H^k(a'_i\comp b)$ is an isomorphism in $\cat{A}$ for every $k\in\ZZ.$ Since
$C$ has only a finite number of non-zero cohomologies, we can fix $k,$
and for simplicity of notation we will denote with an overline the
functor $H^k.$ Now, it is easy to see that the sequence
\[
\xymatrix{
0\ar[r]&\overline{M}\ar[r]^-{\overline{f}}&\overline{M}\ar[r]^-{\overline{a}}&
\overline{C}
}
\]
is exact in $\cat{A}.$ On the other hand, the distinguished triangle
\eqref{eqn:distr} also yields an exact sequence
\[
\xymatrix{
0\ar[r]&\overline{M}\ar[r]^-{\overline{f}}&\overline{M}\ar[r]^-{\overline{a'}}&\overline{N}\ar[r]&0.
}
\]
As $\overline{b}\comp \overline{a'}=\overline{a},$ this
implies that $\overline{b}\colon\overline{N}\to
I:=\im\overline{a}\subseteq\overline{C}$ is an isomorphism. Denoting
moreover $\im\overline{a_i}\subseteq\overline{C}$ by $I_i,$
\eqref{eqn:strangeseq} clearly implies that $I_i\subseteq I_j$ for
$i<j.$ As $\overline{C}$ is noetherian, there exists $n\in\NN$ such
that $I_i=I_n$ for $i\ge n,$ and obviously $I_n=I.$ Then we claim that
$\overline{a'_i\comp b}$ is an isomorphism for $i>n.$ Indeed, this is
equivalent to saying that $\overline{b\comp a'_i\comp b}\colon\overline{N}\to I$ is
an isomorphism. Since $\overline{b\comp a'_i}=\overline{a_i}$ by
\eqref{eqn:comp}, this is true if and only if
$\overline{a_i}\rest{I}\colon I\to I$ is an isomorphism, which follows
easily from the fact that
$\overline{a_i}\comp \overline{a_{i-1}}=\overline{a_{i-1}}$ and $I_i=I_{i-1}=I.$
\end{proof}

As a consequence we get the following.

\begin{cor}\label{cor:semigen}
Let $\cat{A}$ be an abelian category with arbitrary direct sums and let $\cat{S}\subseteq\D[b](\cat{A})$ be a thick full triangulated subcategory whose objects have noetherian cohomology. Let $\fun{F}\colon\cat{S}\longrightarrow\cat{T}$ be a full exact functor to a triangulated category $\cat{T}.$ Then there is an orthogonal decomposition
\[
\cat{S}=\rort{(\ker\fun{F})}\ortdec\ker\fun{F}
\]
and $\fun{F}\rest{\rort{(\ker\fun{F})}}$ is faithful.
\end{cor}

\begin{proof}
As in Remark \ref{rmk:keller}, by part b) of \cite[Thm.\ 3.8]{K}, the category $\cat{S}$ (which is idempotent complete being a thick subcategory of an idempotent complete category) is equivalent to a category of compact objects in the derived category of a dg-category.
Thus Theorem \ref{thm:gen2} and Proposition \ref{prop:idnoeth} give the desired conclusion.
\end{proof}

\begin{remark}\label{rmk:locallynoeth2}
If $\cat{A}$ is a locally noetherian abelian category and $\cat{S}$ is the full subcategory of $\Db(\cat{A})$ consisting of all objects with noetherian cohomology, then, in view of Remark \ref{rmk:locallynoeth}, $\cat{S}$ is automatically a thick triangulated subcategory and Corollary \ref{cor:semigen} applies.
\end{remark}

\section{The geometric case}\label{subsec:geometry}

Let $X$ be a noetherian scheme. We denote by $\D(X)$ the full
subcategory of the derived category of sheaves of $\so_X$-modules
consisting of (unbounded) complexes with quasi-coherent
cohomology. Let $\Db(X)$ be the full subcategory of $\D(X)$
consisting of bounded complexes with coherent cohomology. As $X$ is noetherian, $\Db(X)$ is equivalent to $\D[b](\Coh(X)),$ where $\Coh(X)$ is the abelian category of coherent sheaves on $X$ (see \cite[Cor.\ 2.2.2.2]{Il}). Moreover,
$\Dp(X)$ will be the full subcategory of $\D(X)$ consisting of perfect
complexes. Notice that $\Dp(X)\subseteq\Db(X).$

Now assume that $Z$ is a closed subscheme of $X.$ We denote by
$\D_Z(X)$ the full subcategory of $\D(X)$ consisting of complexes with
cohomology supported on $Z.$ We will also need the following full
subcategories of $\D_Z(X)$:
\[
\begin{split}
&\Db_Z(X):=\D_Z(X)\cap\Db(X),\\
&\Dp_Z(X):=\D_Z(X)\cap\Dp(X).
\end{split}
\]

\begin{prop}\label{prop:catcomp} {\bf (\cite{R}, Theorem 6.8.)}
The category $\D_Z(X)$ is compactly generated and the category
of compact objects $\D_Z(X)^c$ coincides with $\Dp_Z(X).$
\end{prop}

\begin{remark}\label{rmk:qcoh}
The category $\Qcoh(X)$ of quasi-coherent sheaves of $\ko_X$-modules
over a noetherian scheme $X$ is a locally noetherian abelian category and the full subcategory of noetherian objects in $\Qcoh(X)$ is precisely $\Coh(X).$ The same is true in the supported case as well.
\end{remark}

The following result is then a straightforward consequence of Proposition
\ref{prop:idnoeth}.

\begin{prop}\label{prop:noeth}
If $X$ is a noetherian scheme containing a closed subscheme $Z$ and
$\ke\in\Db_Z(X),$ then the endomorphism ring $\End_{\Db_Z(X)}(\ke)$ is
idempotent noetherian.
\end{prop}

Corollary \ref{cor:semigen} (applied to the case $\cat{A}=\Qcoh_Z(X)$)
and Remark \ref{rmk:qcoh} immediately give the following.

\begin{cor}\label{cor:geo}
Let $X$ be a noetherian scheme containing a closed subscheme $Z.$ If $\cat{S}$ is either $\Dp_Z(X)$ or $\Db_Z(X)$ and
$\fun{F}\colon\cat{S}\to\cat{T}$
is a full exact functor to a triangulated category $\cat{T},$ then
there is an orthogonal decomposition
$\cat{S}=\rort{(\ker\fun{F})}\ortdec\ker\fun{F}$ and
$\fun{F}\rest{\rort{(\ker\fun{F})}}$ is faithful.
\end{cor}

Consider now the following rather general result.

\begin{lem}\label{lem:dec}
Let $\cat{T}$ be a compactly generated triangulated category with
arbitrary direct sums such that $\cat{T}^c$ has an orthogonal
decomposition $\cat{T}^c=\cat{S}_1\ortdec\cat{S}_2.$ Then
$\cat{T}$ has an orthogonal decomposition
$\cat{T}=\tilde{\cat{S}}_1\ortdec\tilde{\cat{S}}_2,$ where
$\tilde{\cat{S}}_i,$ for $i=1,2,$ is the localizing subcategory
generated by $\cat{S}_i.$
\end{lem}

\begin{proof}
We first show that $\tilde{\cat{S}}_1$ and $\tilde{\cat{S}}_2$ are
orthogonal. Indeed, if $A\in\cat{S}_1,$ then
$\rort{A}:=\{B\in\cat{T}:\Hom(A,B)=0\}\supseteq\tilde{\cat{S}}_2$
because $\rort{A}$ is localizing, $A$ being compact. On the other
hand, if $B\in\tilde{\cat{S}}_2,$ then
$\lort{B}:=\{A\in\cat{T}:\Hom(A,B)=0\}\supseteq\cat{S}_1$ by what we
have just proved. Since $\lort{B}$ is a localizing subcategory of
$\cat{T},$ this implies that
$\lort{B}\supseteq\tilde{\cat{S}}_1.$ Hence
$\Hom(\tilde{\cat{S}}_1,\tilde{\cat{S}}_2)=0$ and a similar argument
yields $\Hom(\tilde{\cat{S}}_2,\tilde{\cat{S}}_1)=0.$

For $i=1,2,$ the canonical full embedding
$j_i\colon\tilde{\cat{S}}_i\to\cat{T}$ has a right adjoint
$s_i\colon\cat{T}\to\tilde{\cat{S}}_i$ (see, for example, Theorem 8.3.3 and Proposition 8.4.2 in \cite{N2}). This provides a canonical map
$j_1\comp s_1(X)\oplus j_2\comp s_2(X)\to X,$ for any $X\in\cat{T},$ which sits
in a distinguished triangle
\[
\xymatrix{\sh[-1]{C_X} \ar[r] & j_1\comp s_1(X)\oplus j_2\comp s_2(X) \ar[r] & X \ar[r] &
C_X.}
\]
For any compact object $S\in\cat{S}_1,$ applying the functor
$\Hom(j_1(S),\farg)$ to this triangle, we obtain isomorphisms
\[
\Hom(j_1(S),j_1\comp s_1(X)\oplus j_2\comp s_2(X))\isomor\Hom(S,s_1(X))\isomor\Hom(j_1(S), X).
\]
This implies that $\Hom(j_i(S),C_X)=0$ for any compact object
$S\in\cat{S}_i$ and $i=1,2.$  Since $\cat{T}$ is compactly generated
and, by assumption, any compact object of $\cat{T}$ is a direct sum of
objects from $\cat{S}_1$ and $\cat{S}_2,$ we deduce that $C_X=0.$
Hence the map $j_1\comp s_1(X)\oplus j_2\comp s_2(X)\to X$ is an isomorphism.
\end{proof}

We can now apply the previous result to a concrete geometric question.

\begin{cor}\label{cor:orth}
Let $Z$ be a connected closed subscheme of a quasi-compact quasi-separated scheme $X.$ Then the triangulated categories $\Dp_Z(X),$ $\Db_Z(X),$ and $\D_Z(X)$ are indecomposable.
\end{cor}

\begin{proof}
By Lemma \ref{lem:dec} and Proposition \ref{prop:catcomp}, a non-trivial orthogonal decomposition of $\Dp_Z(X)$ induces a non-trivial orthogonal decomposition of $\D_Z(X).$ So it is enough to show that the latter category and $\Db_Z(X)$ are indecomposable. As the proof for these two categories is the same, we will deal only with $\D_Z(X).$

Hence assume that there exists an orthogonal decomposition $\D_Z(X)=\cat{S}_1\ortdec\cat{S}_2.$
Following the strategy in \cite[Example 3.2]{Br}, consider the structure sheaf $\ko_Z$ of the subscheme $Z\subseteq X.$ Since $Z$ is connected, the object $\ko_Z$ is indecomposable in $\D_Z(X)$ and thus it belongs to one of the categories $\cat{S}_i,$ for $i=1,2.$ Without loss of generality, let it belong to $\cat{S}_1.$

For any closed point $z\in Z,$ there is a non-trivial morphism
$\ko_Z\to\ko_z.$ Thus $\ko_z\in\cat{S}_1,$ for all closed point $z\in
Z.$ Finally, consider a perfect complex $\ka\in\Dp_Z(X).$ Take an
affine open subset $U\iso\Spec(A)\subseteq X$ such that the
restriction of $\ka$ to $U$ is a non-trivial object. By definition,
$\ka\rest{U}$ is isomorphic in $\D(U)$ to an object $\kp$
corresponding to a bounded complex of
finitely generated projective $A$-modules $P.$ Set $i$ such that
$H^i(P)$ is the greatest non-trivial cohomology of $P.$ Then
$H^i(P)$ is a finitely generated $A$-module and, by Nakayama's
lemma, there is a non-trivial map $H^i(\kp)\to\ko_z,$ for a closed
point $z\in Z.$ This induces a non-trivial map $\kp\to\ko_z$ and
therefore all perfect complexes belong to $\cat{S}_1.$ This implies
that $\cat{S}_1$ coincides with $\D_Z(X).$
\end{proof}

This result, combined with Corollary \ref{cor:geo}, gives Theorem \ref{thm:main}.

\section{A counterexample}\label{sec:counter}

In this section we provide an example of a full exact and non-trivial
functor $\fun{F}:\cat{T}_1\to\cat{T}_2$ between triangulated
categories such that $\cat{T}_1$ is indecomposable and $\fun{F}$ is
not faithful.

To this end, let $A$ be a commutative algebra over a field $\K$ with
generators $x_1, x_2, \dots$ and with relations $x_j x_i=x_i$ for
$i<j.$ Let $\Mod{A}$ be the category of right $A$-modules and set
$\D(A):=\D(\Mod{A}).$ Denote by $\Dp(A)$ the full subcategory of
$\D(A)$ of perfect complexes, i.e.\ the smallest thick subcategory of
$\D(A)$ containing $A.$

\begin{lem}\label{lem:count1}
The triangulated category $\Dp(A)$ is indecomposable.
\end{lem}

\begin{proof}
Obviously $\Dp(A)\iso\Dp(\Spec(A)).$ By Corollary \ref{cor:orth}, the result follows once we know that $\Spec(A)$ is connected. This, in turn, is equivalent to showing that $A$ does not contain non-trivial idempotents. But this is an easy exercise using the definition of the algebra $A.$
\end{proof}

Denote by $I$ the ideal generated by all $x_i,$ so that
$A/I\iso\K.$ Consider the functor
\[
\fun{G}:\D(A)\longrightarrow\D(\K),\qquad X\longmapsto X\stackrel{\mathbf L}{\otimes}_A\K
\]
and set $\fun{F}:=\fun{G}\rest{\Dp(A)}:\Dp(A)\to\D(\K).$

\begin{lem}\label{lem:count2}
The functor $\fun{F}$ is full.
\end{lem}

\begin{proof}
It is easy to see that the result follows if we prove that the morphisms
\begin{equation}\label{eqn:count1}
\begin{split}
&\Hom_A(A,P)\longrightarrow\Hom_\K(\fun{F}(A),\fun{F}(P))=\Hom_\K(\K,P\stackrel{\mathbf L}{\otimes}_A\K)\\
&\Hom_A(P,A)\longrightarrow\Hom_\K(\fun{F}(P),\fun{F}(A))=\Hom_\K(P\stackrel{\mathbf L}{\otimes}_A\K,\K)
\end{split}
\end{equation}
are surjective, for any $P\in\Dp(A).$ Any perfect complex $P$ is
a direct summand in $\Dp(A)$ of a bounded complex of finitely generated free $A$-modules.
Hence, it is sufficient to consider the case when $P$ itself is quasi-isomorphic to a bounded complex
\[
\cp{Q}=\{Q^t\stackrel{d^{t}}{\longrightarrow}\cdots\longrightarrow Q^{-1}\stackrel{d^{-1}}{\longrightarrow} Q^{0}\stackrel{d^{0}}{\longrightarrow} Q^{1}\stackrel{d^{1}}{\longrightarrow}\cdots\stackrel{d^{s-1}}{\longrightarrow} Q^s\}
\]
of finitely generated free $A$-modules.

Take a morphism $f_1:\K\to Q^0\otimes_A\K$ such that the
composition $(d^0\otimes\K)\comp f_1$ is trivial. Composing with
$A\to\K,$ the morphism $f_1$ induces a map $g_1:A\to Q^0\otimes_A\K$
which, in turn, lifts to $h_1:A\to Q^0.$ Now the element $(d^0\comp
h_1)(1)\in Q^1\iso A^m$ is in $I^m$ and $x_n(d^0\comp
h_1)(1)=(d^0\comp h_1)(1),$ for a sufficiently large $n.$ So setting
$h'_1:=(1-x_n)\comp h_1,$ we get $d^0\comp h'_1(1)=0$ and
$\fun{F}(h'_1)=f_1.$ In particular, the first morphism in
\eqref{eqn:count1} is surjective.

Similarly, to deal with the second morphism in \eqref{eqn:count1}, let
$f_2:Q^0\otimes_A\K\to\K$ be a morphism such that the composition
$f_2\comp(d^{-1}\otimes\K)$ is trivial. Again, composing with the
natural morphism $Q^0=Q^0\otimes_A A\to Q^0\otimes_A\K,$ we get a
morphism $g_2:Q^0\to\K$ which lifts to a morphism $h_2:Q^0\to A.$ For
very large $n,$ define $h'_2:=(1-x_n)\comp h_2$ so that, again,
$h'_2\comp d^{-1}(a_j)=0,$ for all $a_j$ in the set of generators
$a_1,\ldots a_r$ of $Q^{-1}.$ Then $\fun{F}(h'_2)=f_2$ and this
concludes the proof.
\end{proof}

To prove that $\fun{F}$ is not faithful, consider the
non-trivial morphism $x_i:A\to A,$ for $i$ any positive integer. On
the other hand, the morphism $\fun{F}(x_i):\K\to\K$ is the trivial morphism.

%%%%%%%%%%%%%%%%%%%%%%%%%%%%%%%%%%%%%

\bigskip

{\small\noindent {\bf Acknowledgements.}
D.O.\ thanks Simons Center for Geometry and Physics for  hospitality and stimulating atmosphere. Part of this article was written while P.S.\ was visiting the Institut Henri Poincar\'e in Paris whose warm hospitality is gratefully acknowledged. P.S.\ would like to dedicate the paper to his wife Anna and to his daughter Giulia who came to light during the final write-up of the article.}

%%%%%%%%%%%%%%%%%%%%%%%%%%%%%%%%%%%%%

\end{document}